\documentclass[12pt]{amsart}
\usepackage{bbm, amssymb, amsmath, amsfonts}

\usepackage{graphicx, epsfig}

\usepackage{xcolor}

 \makeindex

\def\Bbb{\mathbb}

\DeclareMathOperator{\Hess}{Hess}

\DeclareMathOperator{\supp}{supp}

\def \supp {\text{\rm supp\,}}

\def\bR{{\Bbb R}}

\def\bpm{\begin{pmatrix}}
\def\epm{\end{pmatrix}}

\def\bee{\begin{enumerate}}
\def\ee{\end{enumerate}}

\newtheorem{thm}{Theorem}[section]

\newtheorem{prop}[thm]{Proposition}
\newtheorem{proposition}[thm]{Proposition}
\newtheorem{cor}[thm]{Corollary}

\newtheorem{remark}[thm]{Remark}

\begin{document}


\title[Estimates for convolution operators] {Sharp estimates for convolution operators associated to hypersurfaces in $\mathbb{R}^3$ with height $h\le2$}
\author[Akramov]{Ibrokhimbek Akramov}
\address{Silk Road International University of Tourism and Cultural Heritage,
University Boulevard 17, 140104, Samarkand,  Uzbekistan} \email{{\tt
i.akramov1@gmail.com}}
\medskip
\author[Ikromov]{Isroil A. Ikromov}
\address{Institute of Mathematics named after V.I. Romanovsky,
University Boulevard 15, 140104, Samarkand,  Uzbekistan} \email{{\tt
i.ikromov@mathinst.uz}}



\begin{abstract} {
In this article, we   study the convolution  operator $M_k$  with oscillatory kernel, which is related with solutions to the Cauchy problem for the strictly hyperbolic equations.  The operator  $M_k$ is associated to the characteristic hypersurface $\Sigma\subset \mathbb{R}^3$ of the equation and   the smooth amplitude function, which is homogeneous  of order $-k$ for large values of the argument.
We   study  the convolution operators    assuming  that the support of the corresponding amplitude function is contained in a sufficiently small conic  neighborhood of a given point $v\in \Sigma$  at which the height of the surface is less or equal to two. Such class contains surfaces related to simple and the $X_9, \, J_{10}$  type singularities  in the sense of Arnol'd's classification.  Denoting by $k_p$ the minimal   exponent such that $M_k$ is $L^p\mapsto L^{p'}$-bounded for $k>k_p,$ we show that the number $k_p$ depends on some discrete characteristics of the Newton polygon of a smooth function constructed in an appropriate coordinate system.
}

 \end{abstract}
\maketitle



\thispagestyle{empty}

\setcounter{equation}{0}
\section{Introduction}\label{introduction}
It is well known that  solutions to the Cauchy problem for the strictly hyperbolic equations  up to a smooth function can be written as a sum of convolution operators of the type:
\begin{eqnarray*}\nonumber
\mathcal{M}_k=F^{-1}[e^{it\varphi}a_k]F,
\end{eqnarray*}
where  $F$ is the Fourier transform operator, $\varphi\in
C^{\infty}(\mathbb{R}^{\nu}\backslash\{0\})$
 is homogeneous of order one, $a_{k}\in
C^{\infty}(\mathbb{R}^{\nu}_\xi)$ is a homogeneous function of order $-k$ for large $\xi$ (see \cite{sugumoto94} and \cite{sugumoto98}).

After scaling arguments in the time $t>0$ the operator $\mathcal{M}_k$ is reduced to the following convolution operator:
\begin{equation}\label{convoper}
M_k=F^{-1}[e^{i\varphi}a_k]F.
\end{equation}
Further, we investigate the operator $M_k$.
More generally, we may assume that the amplitude function in $M_k$, defined by \eqref{convoper},  belongs to the space  of the classical symbols of Pseudo-Differential Operators  (PsDO)  of order $-k$, which is  denoted  by $S^{-k}(\mathbb{R}^\nu)$ (see \cite{hor}).
Indeed, it is well known that the PsDO is bounded on $L^p(\mathbb{R}^\nu)$ for $1<p<\infty$, whenever $a_0\in S^0(\mathbb{R}^\nu)$. Consequently, the problem is essentially reduced to the case when $a_k$ is a smooth function, which is homogeneous of order $-k$ for large $\xi$.
Therefore, WLOG  we may assume that an amplitude function is smooth and homogeneous of order $-k$ for large $\xi$.

Let $1\leq p\leq 2$ be a fixed number.  We consider the problem: {\it Find the minimal  number $k(p)$  such that the operator $M_{k}: L^{p}(\mathbb{R}^\nu)\rightarrow
L^{p'}(\mathbb{R}^\nu)$ is bounded for any amplitude function $a_k\in S^{-k}(\mathbb{R}^\nu)$, whenever $k>k(p)$. }

Further, we will assume that the function $\varphi$ preserves sign, i.e. we will suppose that $\varphi(\xi)\neq0$ for any $\xi\in \mathbb{R}^\nu\setminus\{0\}$ and $\nu\ge2$. If $\varphi(\xi^0)=0$ at some point $\xi^0\neq0$ then the additional difficulties arise in the investigation of the corresponding operators. We will work out on such kind of operators elsewhere.   Next, without loss of generality we may and will assume that $\varphi(\xi)>0$ for any $\xi\neq0$.
Since  $\varphi$ is a smooth homogeneous function of order one, then, due to the Euler's homogeneity relation we have:
$$
\sum_{j=1}^{n}\xi_{j}\frac{\partial\varphi(\xi)}{\partial\xi_j}=\varphi(\xi),
$$
 and hence the set $\Sigma$  defined by the following:
 \begin{equation}\label{charaksh}
 \Sigma=\{\xi\in\mathbb{R}^\nu : \varphi(\xi)=1\}
 \end{equation}
  is a smooth  or a real analytic hypersurface provided $\varphi$ is a smooth or a real analytic function in $\mathbb{R}^n\setminus\{0\}$ respectively.

Further, we use notation:
\begin{equation}\label{(1.2)}
k_p(\Sigma):=\inf_{k>0}\{k>0:  M_{k}: L^{p}(\mathbb{R}^\nu)\rightarrow
L^{p'}(\mathbb{R}^\nu)\, \mbox{is bounded for any}\, a_k \in S^{-k}(\mathbb{R}^\nu \}.
\end{equation}
It turns out that the number $k_p(\Sigma)$ depends on some geometric properties of the  hypersurface
$\Sigma.$

M. Sugimoto \cite{sugumoto98} consider the problem for the case when $\Sigma\subset \mathbb{R}^3$ is an analytic   surface having at least one non-vanishing principal curvature at every point and obtain an upper bound for the number $k_p(\Sigma)$.  Further, in the paper \cite{Ikrs2} it was considered the same problem and obtained the exact value of the number $k_p(\Sigma)$ in the case of classes of hypersurfaces in $\mathbb{R}^3$ with at least one non-vanishing principal curvature.

  The natural question is: {\bf How can be characterised the number $k_p(\Sigma)$   for the    of hypersurface $\Sigma$ with vanishing   principal curvatures at a point of $\Sigma\subset \mathbb{R}^3$ ?}

We obtain the exact value of $k_p(\Sigma)$, extending  the results proved by M. Sugimoto,  for arbitrary analytic hypersurfaces  satisfying the condition $h(\Sigma)\le2$ (where $h(\Sigma)$ is the height of the hypersurface introduced in \cite{IMacta})   and  for analogical  smooth hypersurfaces under the so-called   $R-$ condition. Actually, the $R-$ condition can be   defined for any smooth function in terms of Newton polyhedrons (see \cite{IMmon}).

Since $\Sigma$ is a compact hypersurface, then following M. Sugimoto it is enough to consider the local version of the problem.
More precisely, we may assume that the amplitude function
 $a_{k}(\xi)$ is concentrated in a sufficiently small conic neighborhood $\Gamma:=\Gamma(v)$ of a given point $v\in \Sigma$   and
$\varphi(\xi)\in C^{\infty}(\Gamma)$. Let's denote by $S^{-k}_0(\Gamma)$ the space of all classical symbols of PsDO of order $-k$ with support in $\Gamma$. Fixing such a point $v\in \Sigma$, let us define the following local exponent $k_p(v)$ associated to this point:
\begin{equation}\label{local}
k_p(v):=\inf_{k>0}\{k: \exists\Gamma\ni v, \,  M_{k}: L^{p}(\mathbb{R}^3)\mapsto
L^{p'}(\mathbb{R}^3)\, \mbox{is bounded, whenever}  \,a_k\in S^{-k}_0(\Gamma)\}.
\end{equation}

Surely, the inequality $k_p(v)\le k_p(\Sigma)$ holds true  for any $v\in \Sigma$.  We show that  the following relation
\begin{eqnarray*}
k_p(\Sigma)=\sup_{v\in \Sigma}k_p(v)
\end{eqnarray*}
holds.
Moreover, our results yield that $k_p(v)$ is an upper semi-continuous function of $v$. Hence, the "supremum"  attains  and we  have:
 \begin{eqnarray*}
k_p(\Sigma)=\max_{v\in \Sigma}k_p(v).
\end{eqnarray*}

Further, we  use the following standard notation assuming $F$ being  a sufficiently smooth function:
$$
\partial^\gamma F(x):=\partial_1^{\gamma_1}\dots \partial_\nu^{\gamma_\nu}F(x):=\frac{\partial^{|\gamma|}F(x)}{\partial x_1^{\gamma_1}\dots\partial x_\nu^{\gamma_\nu}},
$$ where $\gamma=(\gamma_1,\dots, \gamma_\nu)\in \mathbb{Z}^\nu_+$ is a multiindex,  with $\mathbb{Z}_+:=\{0\}\cup \mathbb{N}$, and $|\gamma|:=\gamma_1+\dots+\gamma_\nu.$

Furthermore, for the sake of being definite  we will assume that
$v=(0,0,1)\in \Sigma\subset \mathbb{R}^3$ and $\varphi(0,0,1)=1$. Then after possible a linear transform in the space $\mathbb{R}^3_\xi$, which preserves the point $v$, we may assume that $\partial_1\varphi(0, 0, 1)=0$ as well as $\partial_2\varphi(0, 0, 1)=0$.
Then, in a sufficiently small  neighborhood of the point  $v$ the hypersurface $\Sigma$ is given as the graph of a smooth function. More precisely, we have:
\begin{equation}\label{defphi}
\Sigma\cap\Gamma=\{\xi\in\Gamma:\varphi(\xi)=1\}=\{
(\xi_1, \xi_2, 1+\phi(\xi_1,\xi_2))\in \mathbb{R}^3:  (\xi_1,\xi_2)\in U\},
\end{equation}
where
$U\subset \mathbb{R}^2$ is a sufficiently small neighborhood of the origin and, $\phi\in C^{\infty}(U)$ is a smooth function satisfying $\phi(0,0)=0,
\nabla\phi(0,0)=0$
(compare with \cite{sugumoto98}) i.e.  $(0, 0)$ is a singular point of the function $\phi$. By a singular point of  a function we mean a critical point of the smooth function (see \cite{agv}).

We can define a height of the hypersurface $\Sigma$ at the point $v$ by $h(v, \Sigma):=h(\phi)$ (see Section \ref{newton} for the definition of a height of a smooth function).The number $h(v, \Sigma)$ can easily be seen to be
invariant under affine linear changes of coordinates in the ambient space $\mathbb{R}^3$ (see \cite{IMacta}).
 Then we define a height of the smooth hypersurface $\Sigma$ by the relation:
$h(\Sigma):=\sup_{v\in \Sigma} h(v, \Sigma)$. It is well known that $h(v, \Sigma)$ is an upper semi-continuous function on the two-dimensional surface $\Sigma$ (see \cite{IMacta}). Thus, actually the "supremum" is attained at a point of the compact set $\Sigma$ and we can write     $h(\Sigma):=\max_{v\in \Sigma} h(v, \Sigma)$ (see \cite{IMacta} for more detailed information).

Surely, similarly one can define $\Sigma$ in a neighborhood of the point $v=(0,\dots, 0, 1)\in \mathbb{R}^\nu$ as
the graph of a smooth function $1+\phi$ defined in a sufficiently small  neighborhood $U$ of the origin of $\mathbb{R}^{\nu-1}$ satisfying the conditions: $\phi(0)=0, \, \nabla\phi(0)=0$.

\section{Newton polyhedrons and adapted and linearly adapted coordinate systems}\label{newton}


In order to formulate our main results we need
 notions of a height and a linear height of a smooth function \cite{varchenko} (also see \cite{IM-adapted}).  Let $\phi$ be a smooth real-valued function defined in a neighborhood of the origin of $\mathbb{R}^2$ satisfying the conditions: $\phi(0)=0$ and $\nabla\phi(0)=0$. Consider
the associated Taylor series
\begin{eqnarray*}
\phi(x) \thickapprox \sum_{|\alpha|=2}^\infty c_{\alpha} x^\alpha
\end{eqnarray*}
of $\phi$ centered at the origin, where $x^\alpha=x_1^{\alpha_1}x_2^{\alpha_2}$.

The set
\begin{eqnarray*}
\mathcal{T}(\phi) := \{\alpha\in  \mathbb{Z}^2_+\setminus\{(0; 0)\}: c_{\alpha}:=\frac1{\alpha_1!\alpha_2!} \partial_1^{\alpha_1}\partial_2^{\alpha_2} \phi(0, 0)\neq0\}
\end{eqnarray*}
is called to be the Taylor support of $\phi$ at the origin. The Newton polyhedron or polygon $\mathcal{N}(\phi)$ of
$\phi$ at the origin is defined to be the convex hull of the union of all the octants $\alpha +\mathbb{ R}_+^2$, with
$\alpha\in \mathcal{T}(\phi)$.  A Newton diagram $\mathcal{D}(\phi)$ is the union of all edges of the Newton polygon.
 We use coordinate $t := (t_1, t_2)$ in the space $\mathbb{R}^2\supset \mathcal{N}(\phi)$.
The Newton distance in the sense of Varchenko \cite{varchenko}, or shorter distance, $d = d(\phi)$ between
the Newton polyhedron and the origin is given by the coordinate $d$ of the point $(d,  d)$ at which
the bi-sectrix $t_1 = t_2$ intersects the boundary of the Newton polyhedron. A principal face is the face of minimal dimension containing the point $(d, d)$.
Let $\gamma\in \mathcal{D}(\phi)$ be a face of the Newton polyhedron. Then the formal  power series (or a finite sum the in case $\gamma$ is a compact edge):
\begin{eqnarray*}
\phi_\gamma\thickapprox \sum_{\alpha\in \gamma} c_\alpha x^\alpha
\end{eqnarray*}
is called to be a part of Newton polyhedron corresponding to the face $\gamma$.

The part of Taylor series of the function $\phi$ corresponding to the principal face, which we denote by $\pi$,   is called to be a principal part $\phi_\pi$ of the function $\phi$.
If there exists a coordinate system for which the principal face is a point then we set $\mathfrak{m}=1$, otherwise $\mathfrak{m}=0$. The number $\mathfrak{m}$ is called to be a multiplicity of the Newton polyhedron. The multiplicity of the Newton polyhedron is well-defined (see  \cite{IMmon}).

 A height of a smooth
function $\phi$ is defined by \cite{varchenko}:
\begin{equation}\label{height}
h(\phi) := \sup \{d_y\},
\end{equation}
where the "supremum" is taken over all local coordinate system $y$ at the origin (it means  a smooth coordinate system defined near the origin which
preserves the origin), and where $d_y$ is the
distance between the Newton polyhedron and the origin in the coordinate $y$.

The coordinate system $y$ is called to be adapted to $\phi$ if $h(\phi)=d_y$, where $d_y$ is the Newton distance in the coordinate $y$. Existence of an adapted coordinate system was proved by Varchenko A.N.  for analytic functions of two variables in the paper \cite{varchenko} (also see \cite{IM-adapted}, where  analogical results are obtained  for smooth functions).

If we restrict ourselves with a linear change of variables, i.e.
\begin{eqnarray*}
h_{lin}(\phi) := \sup_{GL} \{d_y\},
\end{eqnarray*}
where $GL:=GL(\mathbb{R}^2)$ is the group of all linear transforms of $\mathbb{R}^2$,   then we come to a notion of a linear height of the function $\phi$ \cite{IMmon}.

Surely, $h_{lin}(\phi)\le h(\phi)$ for any smooth function $\phi$ with $\phi(0)=0$ and $\nabla\phi(0)=0$. If $h_{lin}(\phi)= h(\phi)$ then we say that the coordinate system $x$ is linearly adapted (LA) to $\phi$. Otherwise, if $h_{lin}(\phi)< h(\phi)$
then the coordinate system is not linearly adapted (NLA) to $\phi$. Note that we can define $h_{lin}(v, \Sigma):=h_{lin}(\phi)$. The notion $h_{lin}(v, \Sigma)$ is well-defined, that is,  it does not depend on linear change of coordinate  in the ambient space $\mathbb{R}^3$ (see \cite{IMmon}).

Further, we  mainly consider the case $h(\phi)\le 2$ although some results hold true for
arbitrary values of $h$.

\section{The main results}\label{hb2}

Further, we use notation
\begin{eqnarray*}
\Hess\phi(x):=
\begin{pmatrix}
\partial_1^2\phi(x)& \partial_1\partial_2\phi(x)\\
\partial_2\partial_1\phi(x)& \partial_2^2\phi(x)
  \end{pmatrix}.
\end{eqnarray*}
The  symmetric matrix $\Hess\phi(x)$ is called to be the Hessian matrix of the function $\phi$ at the point $x$.
The sharp estimates  for the operator $M_k$ in the  case when $\Hess\phi(0,  0) \not= 0$
  have been considered in the previous papers  \cite{sugumoto98} and \cite{Ikrs1}. In this paper we consider the case when $\Hess\phi(0,  0) =0$, more precisely, we assume that $\partial_1^{\alpha_1}\partial_2^{\alpha_2}\phi(0)=0$   for any $\alpha\in \mathbb{Z}_+^2$ with $|\alpha|\le2$  i. e.  the singularity of the function $\phi$ has co-rank two or equivalently rank zero  (see \cite{agv} for a definition of rank of a critical point). It means that both principal curvatures of the surface $\Sigma$ vanish at the point $(0, 0, 1)$.

Further, we need the following results.

\begin{prop}\label{LAcase}
Assume that
$\partial_1^{\alpha_1}\partial_2^{\alpha_2}\phi(0, 0)=0$ for any $\alpha\in \mathbb{Z}_+^2$ with $|\alpha|\le2$,  $h(\phi)\le2$ and   the coordinate system is not linearly adapted to $\phi$, i.e. $h_{lin}(\phi)<h(\phi)$. Then the following statements hold true:\\
\begin{enumerate}
\item[(i)] The function $\phi$ after possible linear change of variables can be written in the following form on
a sufficiently small neighborhood of the origin:
\begin{equation}\label{(1.2.2)}
    \phi(x_1, x_2)=b(x_1, x_2)(x_2-x_1^m\omega(x_1))^2+b_0(x_1),
\end{equation}
where $b,  b_0, \omega$ are smooth functions;\\
\item[(ii)]
The function $b$
 satisfies the conditions: $b(0,  0) =0, \, \partial_1 b(0, 0)\neq0,\, \partial_2 b(0, 0)=0$;\\
\item[(iii)]     $2\le m\in \mathbb{N}$ and $\omega$ is a smooth function satisfying the condition: $\omega(0) \not= 0$;\\
\item[(iv)]either  $b_0(x_1) = x_1^n\beta(x_1)$ with $2m+1<n<\infty$, where $\beta$ is a smooth function with $\beta(0) \not= 0$  (singularity of type $D_{n+1}$),  or  $b_0$ is a flat function (in the case when $b_0$ is a flat function we formally put $n=\infty$, singularity of type $D_\infty$);\\
\item[(v)]  $h(\phi) = 2n/(n+1)$ ( $h(\phi)=2$, when  $n=\infty$) and  $h_{lin}(\phi)=(2m+1)/(m+1)$.
\end{enumerate}
Conversely, if the conditions  (i)-(iv) are fulfilled  then $\partial^{\alpha_1}\partial^{\alpha_2}\phi(0, 0)=0$ for any $\alpha\in \mathbb{Z}_+^2$ with $|\alpha|\le2$ and $h(\phi)\le2$. Moreover, the inequality $h_{lin}(\phi)<h(\phi)$ holds true.
\end{prop}

\begin{remark}
It is easy to show that the numbers $m, n$ are well-defined for arbitrary smooth function $\phi$ having $D$ type singularity (see  \cite{IMmon} also see Proposition \ref{normform}). Thus, to each point $v\in \Sigma$ of surface with $D$ type singularity  we can attach a pair $(m(v), n(v))$ due to the Proposition \ref{normform}.

It should be worth  to note that in the case $\Hess\phi(0, 0)\neq0$  there is one more  class of functions (namely, the class of smooth functions having singularity of type $A$), which has no linearly adapted coordinate system (see \cite{Ikrs1}). So, if $h(\phi)\le2$ and the phase function $\phi$ has no linearly adapted coordinate system then   necessarily it has either $A$ or $D$ type singularities. Actually, the case $A$ was treated in the previous papers \cite{Ikrs1} and \cite{sugumoto98}.
\end{remark}

Further, we shall work under the following  $R-$ condition: If  $\phi$ has singularity of type $D_\infty$ (e.g. if $b_0$ is a flat function at the origin) then $b_0\equiv0$. Surely, if $\phi$ is a real analytic function then   the $R-$ condition is fulfilled automatically (compare with $R-$condition proposed in  \cite{IMmon} for more general smooth functions, which is defined in terms of the Newton polyhedrons).

Our main results are the following:

\begin{thm}\label{main} If $\phi$ is a  smooth function defined by \eqref{defphi} with $h(\phi)\le2$  and rank of singularity at the origin is zero and
$1\le p\le 2$ is a fixed number then
\begin{eqnarray*}
 k_p(v) := \left(6-\frac2h\right)\left(\frac1p-\frac12\right),
\end{eqnarray*}
except the case when $\phi$ has singularity of type $D_{n+1}$ with
 $2m + 1< n$, where $v=(0, 0, 1)$.

Moreover, if the smooth function $\phi$  satisfies the $R-$condition and  has $D_{n+1}$ type singularity at the origin,  with $2m+1<n\le \infty$, then
\begin{eqnarray*}
 k_p(v) :=\max\left\{ \left(5-\frac1{2m+1}\right)\left(\frac1p-\frac12\right), \left(6-\frac{2m+2}{n}\right)\left(\frac1p-\frac12\right)+\frac{2m+1}{2n}-\frac12\right\}.
\end{eqnarray*}

\end{thm}

\begin{prop}\label{LAcase}
If the coordinate system is linearly adapted to $\phi$ that is  $h_{lin}(\phi)=h(\phi)$ then
 the following relation holds true:
\begin{eqnarray*}
k_p(v)= \left(6-\frac2h\right)\left(\frac1p-\frac12\right).
\end{eqnarray*}
\end{prop}

\begin{proof}A proof of the Proposition  \ref{LAcase} follows from Theorems \ref{preest} and \ref{lowest}.
\end{proof}
Note that the statement of the Proposition  \ref{LAcase} holds true for arbitrary smooth function for which there exists a linearly adapted coordinate system. More precisely,  there is  no any restriction  $h(\phi)\le 2$, provided $h_{lin}(\phi)=h(\phi)$. The results of the Proposition \ref{LAcase}
agree with the corresponding results on the Fourier restriction estimates proved in \cite{IM-uniform}.

The Proposition \ref{LAcase} yields the following corollary

\begin{cor}\label{convexcase}
If  the surface $\Sigma\subset \mathbb{R}^3$ is a smooth convex hypersurface then
 the following relation holds true:
\begin{eqnarray*}
k_p(v)= \left(6-\frac2{h(\Sigma)}\right)\left(\frac1p-\frac12\right).
\end{eqnarray*}
\end{cor}

Note that the Corollary \ref{convexcase} improves the results proved in the paper \cite{sugumoto94} (see page no. 522 Theorem 1) in the three-dimensional case.

\begin{remark} Suppose $\phi$ has $D$ type singularity at a point. Then, as noted before, the pair of positive integers  $(m, n)$ is well-defined.  Moreover,  the condition $2m+1 \ge n$ corresponds to the linearly adapted coordinate
system introduced in \cite{IMmon}. It means that in the relation \eqref{height} the "supremum" is attained in a linear change
of variables. So, there exists a linear change of variables y such that $d_y = h(\phi)$ under the condition
$2m + 1\ge n$. Moreover, if $2m + 1 < n$  then for any linear change of variables $y$ we have $d_y < h(\phi)$.
\end{remark}

{\bf Conventions:}  Throughout this article, we shall use the variable constant notation,
i.e., many constants appearing in the course of our arguments, often denoted by
$c, C, \varepsilon, \delta$; will typically have different values at different lines. Moreover, we shall use symbols
such as $\sim, \lesssim;$ or $<<$ in order to avoid writing down constants, as explained in \cite{IMmon} (
Chapter 1). The symbol $\lesssim_g$ means that the constant depends on $g$.  By $\chi_0$ we shall denote a non-negative smooth cut-off function on $\mathbb{R}^\nu$ with typically
small compact support which is identically $1$ on a small neighborhood of the origin.

\section{Preliminaries}

We define the Fourier operator and its inverse by the following \cite{stein-book}:
\begin{eqnarray}\nonumber
F(u)(\xi):=\frac1{\sqrt{(2\pi)^\nu}}\int_{\mathbb{R}^\nu} e^{ i\xi \cdot x}u(x)dx,
\end{eqnarray}
and
\begin{eqnarray}\nonumber
u(x):=\frac1{\sqrt{(2\pi)^\nu}}\int_{\mathbb{R}^\nu} e^{-i \xi\cdot x} F(u)(\xi)d\xi
\end{eqnarray}
respectively for a Schwartz function $u$, where $\xi\cdot x$ is the  usual inner product of the vectors $\xi$ and $x$. Then the  Fourier transform and inverse Fourier transform of a distribution are  defined by the standard arguments.

Note that the  boundedness problem for the convolution operators is related to behaviour  of the following convolution kernel:
\begin{eqnarray}\nonumber
K_k:=F^{-1}(e^{i\varphi}a_k),
\end{eqnarray}
which is defined as the inverse Fourier transform of the corresponding distribution.

It is well known that (see \cite{sugumoto98}) the main contribution to $K_k$ gives points $x$ which belongs to a sufficiently small neighborhood of the set $-\nabla \varphi(\supp(a_k)\setminus\{0\})$.

In the paper \cite{sugumoto98} it had been shown a relation between the boundedness of the convolution operator $M_{k}$ and behaviour of the following oscillatory integral:
\begin{eqnarray*}\nonumber
I(\lambda, s)=\int_{\mathbb{R}^{\nu-1}}e^{i\lambda(\phi(x)+s\cdot x)}g(x)dx,
 (\lambda>0,\, z\in \mathbb{R}^{\nu-1}),
\end{eqnarray*}
where $ g\in C_{0}^{\infty}(U)$ and $U$  is a sufficiently small neighborhood of the origin.

More precisely,  the following statements  are proved \cite{sugumoto98}:

 \begin{proposition} \label{Sugi1}
 Let $q\geq 2$ and $\gamma\geq 0$.  Suppose for all $g\in C_{0}^ {\infty}(U)$ and $\lambda>1$,
\begin{equation}\label{averbound}
\|I(\lambda,\cdot)\|_{L^{q}(\mathbb{R}_{s}^{\nu-1})}\lesssim_g\lambda^{-\gamma}.
\end{equation}
 Then
$K_{k}(\cdot):=F^{-1}[e^{i\varphi }a_{k}](\cdot)\in
L^{q}(\mathbb{R}^{\nu})$ and $M_{k}:L^{p}(\mathbb{R}^{\nu}) \rightarrow
L^{p'}(\mathbb{R}^{\nu})$ is the bounded  operator for $p=\frac{2q}{2q-1}$, if
$k>\nu-\gamma-\frac{1}{q}$.
 \end{proposition}

Besides,  M. Sugimoto proved a version of  Proposition \ref{Sugi1} corresponding to the case $q=\infty$. One can define
\begin{eqnarray*}\nonumber
K_{k, j} (x)=F^{-1}[e^{i\varphi}a_k  \Phi_j ](x).
\end{eqnarray*}
Here $\{\Phi_j \}_{j=1}^\infty$  is a Littlewood-Paley partition of unity which is used to
define the norm
\begin{eqnarray*}\nonumber
\|f\|_{B^s_{
p, q}}:= \left(\sum_{j=0}^\infty
(2^{ js} \|F^{-1}(\Phi_j F(f)) \|_{L^p})^q\right)^\frac1q
\end{eqnarray*}
of Besov's space $B^s_{p, q}$ (see \cite{Bergh}).

\begin{proposition}\label{Sugi2}
 Let  $\gamma\ge0$. Suppose, for all $g\in C^\infty_0(U)$ and  $\lambda>1$,
 \begin{equation}\label{vander}
\|I(\lambda,  \cdot)\|_{L^\infty(\mathbb{R}^{\nu-1}_z)}\lesssim_g  \lambda^{-\gamma},
\end{equation}
where $C_g$ is independent of $\lambda$. Then $\{K_{k, j} ]\}_{j=1}^\infty$ is bounded in $L^\infty(\mathbb{R}^\nu)$, if
$k=\nu-\gamma$. Hence $M_k$ is $L^p(\mathbb{R}^\nu)\mapsto L^{p'}(\mathbb{R}^\nu)$ bounded, if $k>(2\nu-2\gamma)(\frac1{p}-\frac12)$. This
inequality can be replaced by an equation, if $p\neq1$.
\end{proposition}

\subsection{On the  linearly adapted coordinate system}

In this subsection we give a proof of  the Proposition \ref{LAcase}.

Let $P$ be a weighted  homogeneous polynomial. By $\mathfrak{n}(P)$ we denote the maximal order of vanishing of $P$ along the unit circle $S^1$ centered at the origin.

We use the following Proposition:

\begin{prop}\label{normform}

 Assume that
$\partial^{\alpha_1}\partial^{\alpha_2}\phi(0, 0)=0$ for any multi-index $\alpha:=(\alpha_1, \alpha_2)\in \mathbb{Z}_+^2$ with $|\alpha|:=\alpha_1+\alpha_2\le2$. Then the following statements hold:\\
\begin{enumerate}
\item[(a)] If $\phi_3$,  which is the homogeneous part of degree $3$ of the Taylor polynomial of $\phi$,  satisfies the condition
$\mathfrak{n}(\phi_3) < 3$, then $\phi$, after possible linear change of variables, can be written in the following form on
a sufficiently small neighborhood of the origin:
\begin{equation}\label{(1.2.2)}
    \phi(x_1,x_2)=b(x_1,x_2)(x_2-\psi(x_1))^2+b_0(x_1),
\end{equation}
where $b,  b_0$ are smooth functions, and $b(0, 0) =0,\, \partial_1 b(0, 0)\neq0, \partial_2 b(0, 0)=0$  and also $\psi(x_1) = x_1^m\omega(x_1)$
 with  $m\ge2$ and $\omega(0) \not= 0$ unless $\psi$  is a flat
function.

Moreover, either\\
\item[(ai)] $b_0$ is flat, (singularity of type $D_\infty$) and $h(\phi) = 2$;
or\\
\item[(aii)] $b_0(x_1) = x_1^n\beta(x_1)$ with $n\ge 3$, where $\beta(0) \not= 0$ (singularity of type $D_{n+1}$) and $h(\phi) = 2n/(n+1)$.
\end{enumerate}
In these cases we say that $\phi$ has singularity  of type $D$.\\
\begin{enumerate}
\item[(b)] If $\mathfrak{n}(\phi_3) = 3$ and $h(\phi) \le2$, then, $\phi$, after a possible linear transformation,  can be written
as follows:\\
$$\phi(x_1,  x_2) = b_3(x_1,  x_2)(x_2 -\psi(x_1))^3+ x_2x_1^{k_1} b_1(x_1) + x_1^{k_0}b_0(x_1);
$$
where $b_3,  b_1,  b_0 $ are smooth functions, $k_0\ge4, k_1\ge3$, also $\psi(x_1) = x_1^m\omega(x_1)$
 with  $m\ge2$ and $\omega(0) \not= 0$ unless $\psi$  is a flat
function. Moreover, $b_3(0, 0)\neq0$ and either\\
\item[(bi)] $k_0 = 4$ with $b_0(0)\neq 0$ and $k_1\ge 4$ this is $E_6$ type singularity and $h(\phi) = 12/7$; or\\
\item[(bii)] $k_1 = 3$ with $b_1(0) \neq 0$ and $k_0\ge 5$ this is $E_7$ type singularity and $h(\phi) = 9/5$; or\\
\item[(biii)] $k_0 = 5$ with $b_0(0)\neq 0$ and $k_1\ge 4$ this is $E_8$ type singularity and $h(\phi) = 15/8$.\\
In these cases we say that $\phi$ has singularity of type $E$.\\
\item[(biv)] Either $k_0=6$ with $b_0(0)\neq0$ and $k_1\ge4$ or  $k_1=4$ with $b_1(0)\neq0$ and $k_0\ge6$; or \\
\item[(c)] $\phi_2=\phi_3\equiv0$ and $\phi_4\not\equiv0$ with $\mathfrak{n}(\phi_4)\le2$.
\end{enumerate}
 In the cases (biv) and c) we have   $h_{lin}(\phi)=h(\phi)=2$.
\end{prop}

\begin{remark}
In the case (biv) the function $\phi$ has singularity of type $J_{10}$ provided the principal part of the function $\phi$ which corresponds to the edge living  on the line $t_1/3+t_2/6=1$ has isolated critical point at the origin. Otherwise, the multiplicity $\mathfrak{m}$ of the Newton polyhedron equal to $1$,  provided $h(\phi)\le2$ . If $\phi_2=\phi_3\equiv0$ and $\mathfrak{n}(\phi_4)\le1$ then the function has $X_9$ type singularity at the origin provided that multiplicity of the unique  critical point   of $\phi_4$ is equal to $9$ (see \cite{agv},  page no. 192).
\end{remark}

Note that the Proposition \ref{normform} can be proved by using implicit function Theorem (see \cite{IMmon}).

A proof of the Proposition \ref{LAcase} is based on the Proposition  \ref{normform}.
It should be noted that $1\le \mathfrak{n}(\phi_3)\le 3$, whenever $\phi_3$ is a nontrivial polynomial. Note that if $\mathfrak{n}(\phi_3)=1$ then the function has $D_4^\pm$ type singularity and the coordinate system is linearly adapted to $\phi$ (see \cite{BIM22}) and  $h_{lin}(\phi)=h(\phi)=\frac32$. Thus, if the coordinate system is not linearly adapted  to $\phi$  then necessarily we have $\mathfrak{n}(\phi_3)\ge2$.
Assume $\mathfrak{n}(\phi_3)=2$. Then the function $\phi$ has $D_{n+1}^\pm$ type singularity at the origin with $4\le n\le\infty$. In this case the coordinate system is linearly adapted to $\phi$ if and only if $2m+1\ge n$ (see \cite{IMmon}).  Moreover, the conditions (i)-(iv) imply:  $\partial^{\alpha_1}\partial^{\alpha_2}\phi(0, 0)=0$ for any $\alpha$ with $|\alpha|\le2$, $h(\phi)\le2$ and  $h_{lin}(\phi)<h(\phi)$. In fact,  the last inequality is equivalent to  $2m+1<n$.

Assume $\mathfrak{n}(\phi_3)=3$ and $h(\phi)\le2$. Then it is well-known that the coordinate system is  linearly adapted to $\phi$  (see \cite{BIM22}) which contradicts to the conditions of the Proposition  \ref{LAcase}. Indeed, if the coordinate system is not linearly adapted then $\kappa_2/\kappa_1=:m\in \mathbb{N}$ provided $\kappa_2\ge\kappa_1$ (see \cite{IMmon}).  So, if $\mathfrak{n}(\phi_3)=3$ then after possible linear change of variables we may assume that $\phi_3(x_1, x_2)=x_2^3$.
Consider the supporting line $\kappa_1t_1+\kappa_2t_2=1$ with $\kappa_2=1/3$ to the Newton polyhedron $\mathcal{N}(\phi)$.  Then obviously $\kappa_1<\frac13$. Note that $\kappa_1\ge\frac16$, otherwise $2<\frac1{|\kappa|}= h_{lin}(\phi)\le h(\phi)$, where and furthermore we use notation: $|\kappa|:=\kappa_1+\kappa_2$. Consequently, $\kappa_1=\frac16$.  Then simple arguments show that the principal face $\pi$  lies on the line $\frac{t_1}6+\frac{t_2}3=1$.  Therefore $\mathfrak{n}(\phi_\pi)=3$ under the condition that the coordinate system is not linearly adapted. Then Varchenko algorithm shows that  $h(\phi)>2$ (see \cite{varchenko} and \cite{IM-adapted}).  Thus, under the condition $h(\phi)\le 2$ the statement (biv) holds true.

Now, assume that $\phi_2=\phi_3\equiv0$. Then we claim that the coordinate system is linearly adapted to $\phi$ under the condition $h(\phi)\le2$.  Note that if  $\phi_2=\phi_3=\phi_4\equiv 0$ then the Newton polyhedron is contained in the set $\{t: t_1/5+t_2/5\ge1\}$ and hence  $h_{lin}(\phi)\ge 5/2>2$ which contradicts to the assumption $h(\phi)\le2$ of the Proposition  \ref{LAcase}.

Thus, we may assume that $\phi_4\not\equiv0$ under the conditions $\phi_2=\phi_3\equiv0$ and $h(\phi)\le2$. Then we have $0\le \mathfrak{n}(\phi_4)\le4$.

It is well-known that if $\mathfrak{n}(\phi_4)\le2$ then the coordinate system is adapted to $\phi_4$, hence  also  to $\phi$ (see \cite{varchenko} and also \cite{IM-adapted} Theorem 3.3). Thus in this case the coordinate system is linearly adapted to $\phi$ i.e. $h_{lin}(\phi)=h(\phi)$ which contradicts to the assumptions of the Proposition \ref{LAcase}.

Finally, assume that $\mathfrak{n}(\phi_4)\ge3$ then we claim that $2<  h_{lin}(\phi)\le h(\phi)$.

First, suppose $\mathfrak{n}(\phi_4)=3$. Then, after possible linear change of variables,  the Newton polyhedron contains the point $(1, 3)$ and there is no any other point of  $\mathcal{N}(\phi)$ on the line $\{t: t_1+t_2=4\}$.
Hence, there exists a supporting line $L:=\{(t_1, t_2): \kappa_1 t_1+\kappa_2 t_2=1\}$ associated to a pair $(\kappa_1, \kappa_2)$ satisfying the conditions $\kappa_2\ge \kappa_1$ with $\kappa_1<1/4$,  and  $\kappa_1+3\kappa_2=1$, the last relation means that the point $(1, 3)\in L$. Then it is easy to see that $\kappa_1+\kappa_2<1/2$. Hence $2< 1/|\kappa|\le h_{lin}(\phi)\le h(\phi)$.

If $\mathfrak{n}(\phi_4)=4$, then after possible linear change of variables,   the line $\{t: t_1+t_2=4\}$ does not contain  any point of the Newton polyhedron $\mathcal{N}(\phi)$ but, the point $(0, 4)$. Then there is a supporting line $\{(t_1, t_2): \kappa_1 t_1+\kappa_2 t_2=1\}$ with $\kappa_2=1/4$ and $\kappa_1<1/4$. Therefore $\kappa_1+\kappa_2<1/2$. Hence $2< 1/|\kappa|\le h_{lin}(\phi)\le h(\phi)$. Therefore, if $\mathfrak{n}(\phi_4)>2$ then $h(\phi)>2$ under the conditions $\phi_2=\phi_3\equiv0$. Thus, we have $\mathfrak{n}(\phi_4)\le2$, whenever the conditions of the Proposition \ref{LAcase} are fulfilled.

Thus,  we come to the conclusion:  if $\phi_2=\phi_3\equiv0$ and $h(\phi)\le2$ then the coordinate system is linearly adapted to the function $\phi$.

Thus the Proposition \ref{LAcase} is proved.

\section{An upper bound for the critical exponent}

In this section we obtain  an upper bound for the critical exponent.

\begin{thm}\label{preest}
Let $\Sigma\subset \mathbb{R}^3$ be a smooth surface  given as the graph \eqref{defphi}  of a smooth function $1+\phi$ satisfying the conditions $\phi(0)=0$ and $\nabla\phi(0)=0$ in a  neighborhood of the point $v:=(0, 0, 1)$. Then  the following estimate holds true:
\begin{equation}\label{upperbound}
k_p(v)\le \left(6-\frac2{h(\phi)}\right)\left(\frac1p-\frac12\right).
\end{equation}
 \end{thm}

Note that there is no restriction $h(\phi)\le 2$ in the Theorem \ref{preest}.  In particular,  the upper bound \eqref{upperbound} holds true for the case when the function $\phi$ has $D$ type singularity and the estimate does not depend on the number $m$. It turns out that, it is the sharp bound for  the $k_p(v)$ under the condition $2m+1\ge n$ when the phase has $D_{n+1}$ type singularities. Also, it is the sharp bound in the case $n=\infty=m$.

\begin{cor}\label{unicor}
 Let $\Sigma\subset \mathbb{R}^3$ be a smooth surface defined by \eqref{charaksh}, with  a smooth function $\varphi$ with $\varphi(\xi)>0$  for any $\xi\neq0$ then  the following estimate holds true:
\begin{equation}\label{genbound}
k_p(v)\le \left(6-\frac2{h(\Sigma)}\right)\left(\frac1p-\frac12\right),
\end{equation}
\end{cor}

\begin{proof}
A proof of the Theorem \ref{preest} is based on uniform estimates for the Fourier transform of the surface-carried measures.
Remark that the upper bound \eqref{upperbound} does not depend whether the coordinate system are linearly adapted to $\phi$ or not.

 \subsection{Uniform estimates}

Due
to the uniform with respect to parameters $s$   estimate for the oscillatory integral  and  Proposition \ref{Sugi2}  we obtain an upper bound for the number $k_p$.

Indeed,  without loss of generality we may assume
$\Sigma$ is given as the graph of a smooth function $\{(y_1,  y_2,  1 + \phi(y_1,  y_2))\}$,  in a
neighborhood of the point $v=(0, 0, 1)$. Moreover, we suppose $\phi(0, 0) = 0$ and $\nabla\phi(0, 0) = 0$. Then the height of
the surface $\Sigma$ at the point $v$ is defined by  the height of the function $\phi$. Hence, by the results of
the paper \cite{IM-uniform} (see \cite{duistermaat} and also \cite{karpushkin} for more general results for oscillatory integrals with analytic phases) we can write:
\begin{eqnarray*}
|I(\lambda, s)|=\left|\int_{\mathbb{R}^2}g(x)e^{i\lambda (\phi(x_1, x_2)+x_1s_1+x_2s_2)}dx\right|\lesssim_g \frac{\log (2+|\lambda|)^\mathfrak{m}}{|\lambda|^\frac1{h(\phi)}},
\end{eqnarray*}
where $\mathfrak{m}=1, 0$ is the multiplicity of the Newton polyhedron.

If $\mathfrak{m} = 0$, then from Proposition \ref{Sugi2}, proved by M. Sugimoto,
we have the upper bound \eqref{upperbound} for the $k_p(v)$.

Suppose $\mathfrak{m}=1$ then for any positive real number $\varepsilon$ we have
\begin{eqnarray*}
\left|\int_{\mathbb{R}^2}g(x)e^{i\lambda (\phi(x_1, x_2)+x_1s_1+x_2s_2)}dx\right|\lesssim_{g, \varepsilon} \frac{1}{|\lambda|^{\frac1{h(\phi)+\varepsilon}}}.
\end{eqnarray*}
Then again by using  Proposition \ref{Sugi2}
we obtain the following upper bound for the $k_p(v)$:
\begin{eqnarray*}
k_p(v)\le \left(6-\frac2{h(\phi)+\varepsilon}\right)\left(\frac1p-\frac12\right).
\end{eqnarray*}
Since $\varepsilon$ is any positive number, then the last estimate implies the bound \eqref{upperbound}.

 Theorem \ref{preest}  is proved.
\end{proof}

A proof of the Corollary \ref{unicor} follows from Theorem \ref{preest}. Because, $h(v, \Sigma)$ is an upper semi-continuous function. Then the result follows from the results of the papers \cite{IMacta} and \cite{IM-uniform}.

Further, we consider an upper bound for the number $k_p(v)$ for the case when coordinate system is not linearly adapted to $\phi$.

\subsection{Non-linearly adapted case}

Assume that the coordinate system is not linearly adapted to $\phi$. Then thanks to Proposition \ref{LAcase} the function $\phi$ has $D$ type singularities, under  condition that the singularity of the function has rank zero at the origin.

Let $\phi$ be a function with a singularity of type $D_{n+1} (3\le n\le \infty)$ at the origin  and $m$ is the number  defined by \eqref{(1.2.2)} satisfying the condition $2m+1<n$. Since $m\ge2$ then $5<n$, so $n\le6$.
Consider the following Randol's maximal function (compare with \cite{randol}):
\begin{equation}\label{randold}
M_{m}(s):=\sup_{|\lambda|>1}|\lambda|^{\frac12+\frac1{m+1}}|I(\lambda, s)|.
\end{equation}

\begin{thm}\label{anasug}
Suppose $2m+1<n\le \infty$,  and $\phi$ is a smooth function satisfying the $R-$condition then there exists a neighborhood $U$ of the origin such that for any $a\in C_0^\infty(U)$ the following inclusion:
\begin{equation}\label{estran}
M_{m}\in L_{loc}^{2m+2-0}(\mathbb{R}^2):=\bigcap_{1\le q<2m+2}L_{loc}^{q}(\mathbb{R}^2)
\end{equation}
holds true.
\end{thm}

\begin{proof}
A proof of the Theorem \ref{anasug} follows from more general results of the paper (see \cite{akrikr} Theorem 4.2 and also  \cite{akr}).
\end{proof}


From Theorem  \ref{anasug} it follows the required upper bound for the number $k_p(v)$ in the case $2m+1<n$.
Indeed, first, we use  Proposition \ref{Sugi1} and  obtain $L^{p_0}\mapsto L^{p'_0}$ boundedness of the convolution operator $M_k$ with $k>\frac52-\frac3{2m+2}$ for $p_0=\frac{4m+4}{4m+3}$. Also, we get $L^{p_1}\mapsto L^{p'_1}$ boundedness of the convolution operator with $k>\frac52-\frac1{2n}$ for $p_1=1$ and also $L^{p_2}\mapsto L^{p'_2}$ boundedness of the convolution operator with $k=0$ for $p_2=2$.
Then by the interpolation Theorem for analytic  family of operators (see \cite{stein2}, \cite{Bergh})  we get the required upper bound for the number $k_p(v)$:
\begin{eqnarray*}
 k_p(v) \le \max\left\{ \left(5-\frac1{2m+1}\right)\left(\frac1p-\frac12\right), \left(6-\frac{2m+2}{n}\right)\left(\frac1p-\frac12\right)+\frac{2m+1}{2n}-\frac12\right\}.
\end{eqnarray*}

Further, we consider a lower bound for the number  $k_p(v)$.

 \section{The lower bound for the critical exponent}

\begin{thm}\label{lowest}
Let $\phi$ be a smooth function satisfying the condition of  Theorem \ref{preest}. Then the following lower estimate holds true:
\begin{eqnarray*}
k_p(v)\ge \left(6-\frac2{h_{lin}(\phi)}\right)\left(\frac1p-\frac12\right).
\end{eqnarray*}
 \end{thm}

In this section we reduce a proof of the Theorem \ref{lowest}. The test functions, used in the course of the proof,  are similar  to Knapp type sequence .

\begin{proof}  Let $\phi$ be the phase function and $\phi_\pi$  be the principal part, which is a weighted homogeneous polynomial  with weight $(\kappa_1, \kappa_2)$ provided $0<\kappa_1\le \kappa_2$.  It means that the relation $\phi_\pi(t^{\kappa_1}x_1, t^{\kappa_2}x_2)=t\phi_\pi(x_1, x_2)$ holds for any $x\in \mathbb{R}^2$ and $t>0$.

The case when $\phi_\pi$ is a formal power series will be proved by similar arguments.
Indeed, in this case we have $\kappa_1=0$. Then we consider the part of the function corresponding to the weight $(\varepsilon, \kappa_2^\varepsilon)$, where $(\varepsilon, \kappa_2^\varepsilon)$ is a weight satisfying $(\varepsilon, \kappa_2^\varepsilon)\to (0, \kappa_2)$ as $\varepsilon\to+0$.

Further, suppose that $0<\kappa_1\le \kappa_2$.
 Actually, we show that the modified  sequence of functions suggested by M. Sugimoto in the paper \cite{sugumoto98} can be used to prove sharpness of the upper for $k_p(v)$ in the case when the coordinate system is  linearly adapted.

  Let us take a smooth function in $\mathbb{R}^3$ such that  $a_k(\xi)=|\xi|^{-k}$ for large $\xi$.
For example, we can take $a_k(\xi)=(1-\chi_0(\xi))|\xi|^{-k}$, where $\chi_0$ is a smooth function such that $\chi_0(\xi)\equiv 1$  in a neighborhood of the origin say for $|\xi|\le1$ and $\chi_0(\xi)\equiv0$ for $|\xi|\ge2$.

Following,  M. Sugimoto we introduce the function:
  $G(y)=1+\phi(y) -y\nabla\phi(y)$.
  Define a non-negative smooth function with
   $\chi_0(0)=1$ concentrated in a sufficiently small neighborhood of the origin, and a non-negative smooth function, satisfying $\chi_1(1)=1$,  with support in a sufficiently small neighborhood of the point  $1$ and $\chi_1(\xi)\equiv0$ in a neighborhood of the origin.

We set
\begin{eqnarray*}\nonumber
  u_{j}(x)=2^{j(3-|\kappa|)\left(-\frac{1}{p'}\right)}F^{-1}
  (v_{j}(2^{-j}\xi))(x),
\end{eqnarray*}
where
\begin{eqnarray*}\nonumber
  v_{j}(\xi)=\frac{\chi_0\left(2^{ \kappa_1 j}\frac{\xi_{1}}{\varphi(\xi)}\right)
 \chi_0\left(2^{\kappa_2 j}\frac{\xi_{2}}{\varphi(\xi)}\right)\chi_1(\varphi(\xi))|\xi|^k}
  {\varphi(\xi)^{2}G\left(\frac{\xi_1}{\varphi(\xi)},\frac{\xi_2}{\varphi(\xi)}\right)}\in C^\infty_0(\mathbb{R}^3).
\end{eqnarray*}
Note that $\supp(v_j)$ does not contain the origin, because  $\chi_1(\varphi(\xi))\equiv0$ in a neighborhood of the origin.

The sequence   $\{u_{j}\}_{j=1}^\infty$ is   bounded  in  $L^{p}(\mathbb{R}^{3})$.
  Indeed,  we have:
  \begin{eqnarray*}\nonumber
  u_j(x)=\frac{2^{\frac{3j}p+\frac{|\kappa|j}{p'}}}{\sqrt{(2\pi)^3}} \int_{\mathbb{R}^3} e^{-i2^j(\xi, x)}v_j(\xi)d\xi.
  \end{eqnarray*}

On the other hand
  following M. Sugimoto we use change of variables  $\xi=(\lambda
y,\lambda(1+\phi(y)))$  and get:
 \begin{eqnarray*}\nonumber
  u_{j}(x)=\frac{2^{\frac{3j}p+\frac{|\kappa|j}{p'}}}{\sqrt{(2\pi)^3}} \int_{\mathbb{R}^3}
e^{-i2^{j}\lambda(x_{1}y_{1}+x_{2}y_{2}+x_{3}(1+\phi(y)))}
\chi_0(2^{j\kappa_1}y_1)\chi_0(2^{j\kappa_2}y_2)\\ \chi_1(\lambda)\lambda^k(y_1^2+y_2^2+(1+\phi(y_1, y_2)^2)^\frac{k}2 d\lambda dy.
  \end{eqnarray*}
Finally, we use scaling $2^{j\kappa_1}y_1\mapsto y_1, \,
2^{j\kappa_2}y_2\mapsto y_2$ in variables $y$
and obtain:
 \begin{eqnarray*}\nonumber
  u_{j}(x)=\frac{2^{\frac{3j}p-\frac{|\kappa|j}{p}}}{\sqrt{(2\pi)^3}}   \int
e^{-i2^{j}\lambda(x_{1}2^{-\kappa_1j}y_{1}+x_{2}2^{-\kappa_2j}y_{2}+x_{3}(1+\phi(\delta_{2^{-j}}(y))))}
\chi_0(y_1)\chi_0(y_2)\\ \chi_1(\lambda)\lambda^k(2^{-2\kappa_1j}y_1^2+2^{-2\kappa_2j}y_2^2+(1+\phi(\delta_{2^{-j}}(y))^2)^\frac{k}2 d\lambda dy.
  \end{eqnarray*}

Note that $|2^j\partial^\alpha\phi(\delta_{2^{-j}}(y))|<<1$ as $j>>1$ for $|\alpha|\ge0$ provided the support of $\chi_0$ are small enough.
 If $|x_3|>|x_{1}2^{-\kappa_1j}|+ |x_{2}2^{-\kappa_2j}|$ then we can use integration by parts formula in $\lambda$ and get:
 \begin{eqnarray*}\nonumber
  |u_{j}(x)|\lesssim_N \frac{2^{\frac{3j}p-\frac{|\kappa|j}{p}}}{(1+|x_32^j|)^N},
\end{eqnarray*}
provided $|x_32^j|>>1$, otherwise e.g. if $|x_32^j|\lesssim1$, then the last estimate trivially holds, for the function $u_j(x)$.

Assume $|x_3|\le |x_{1}2^{-\kappa_1j}|+ |x_{2}2^{-\kappa_2j}|$.  Then by using integration by parts formula in $(y_1, y_2)$ variables, we get the following estimate:
 \begin{eqnarray*}\nonumber
  |u_{j}(x)|\lesssim_N \frac{2^{\frac{3j}p-\frac{|\kappa|j}{p}}}{(1+|x_{1}2^{(1-\kappa_1)j}|+ |x_{2}2^{(1-\kappa_2)j}|)^N}.
\end{eqnarray*}
Finally, combining the obtained estimates we obtain:
 \begin{eqnarray*}\nonumber
  |u_{j}(x)|\lesssim_N \frac{2^{\frac{3j}p-\frac{|\kappa|j}{p}}}{(1+|2^jx_3|+|x_{1}2^{(1-\kappa_1)j}|+ |x_{2}2^{(1-\kappa_2)j}|)^N}.
\end{eqnarray*}

 Consequently,
 \begin{eqnarray*}\nonumber
 \|u_{j}\|_{L^{p}}\lesssim 1,\quad \mbox{for}\quad j>>1.
 \end{eqnarray*}
 Hence,  the sequence  $\{u_{j}\}_{j=1}^{\infty}$ is bounded in the space  $L^{p}(\mathbb{R}^3)$.

 On the other hand we have the relation:
 \begin{eqnarray*}\nonumber
  M_{k}u_{j}(x)=2^{j(3-|\kappa|)(-\frac{1}{p'})-kj+2j}
  F^{-1}
  \left(e^{i\varphi(\xi)}\frac{\chi_0\left(2^{j\kappa_1}\frac{\xi_1}
  {\varphi(\xi)}\right)\chi_0\left(2^{j\kappa_2}
  \frac{\xi_2}{\varphi(\xi)}\right)\chi_1(2^{-j}\varphi(\xi))}
  {\varphi(\xi)^{2}G\left(\frac{\xi_1}{\varphi(\xi)},\frac{\xi_2}{\varphi(\xi)}\right)}\right)(x).
   \end{eqnarray*}

 We perform  change of variables given by the scaling  $2^{-j}\xi\mapsto\xi$ and obtain:
\begin{eqnarray*}\nonumber
M_{k}u_{j}(x)=\frac{2^{j((3-|\kappa|)(-\frac{1}{p'})-k+3)}}{\sqrt{(2\pi)^3}}
\int_{\mathbb{R}^3} e^{i2^{j}(\varphi(\xi)- x\xi)}\frac{\chi_0\left(2^{j\kappa_1}\frac{\xi_1}
{\varphi(\xi)}\right)\chi_0\left(2^{j\kappa_2}\frac{\xi_2}{\varphi(\xi)}\right)\chi_1(\varphi(\xi))}
{\varphi^{2}(\xi)G\left(\frac{\xi_1}{\varphi(\xi)},\frac{\xi_2}{\varphi(\xi)}\right)}d\xi.
 \end{eqnarray*}

Then following M. Sugimoto we use change of variables  $\xi=(\lambda
y,\lambda(1+\phi(y)))$  and gain the relation:
\begin{eqnarray*}
M_{k}u_{j}(x)=\frac{2^{j((3-|\kappa|)(-\frac{1}{p'})-k+3)}}{\sqrt{(2\pi)^3}}\int
e^{i2^{j}\lambda(1-(x_{1}y_{1}+x_{2}y_{2}+x_{3}(1+\phi(y))))}\\
\chi_0(2^{j\kappa_1}y_1)\chi_0(2^{j\kappa_2}y_2)\chi_1(\lambda)d\lambda dy.
\end{eqnarray*}

Finally, we use change of variables
$2^{j\kappa_1}y_1\mapsto y_1, \,
2^{j\kappa_2}y_2\mapsto y_2$ and obtain:
\begin{eqnarray*}\nonumber
M_{k}u_{j}(x)=2^{j((3-|\kappa|)(-\frac{1}{p'})-k-|\kappa|
+3)}\int_{\mathbb{R}^3}
e^{2^{j}i\lambda((x_3-1)-2^{-j\kappa_1}y_{1}x_{1}-2^{-j\kappa_2}y_{2}x_{2}
-x_{3}\phi(2^{-j\kappa_1}y_1,
2^{-j\kappa_2}y_2))}\\ \chi_0(y_{1})\chi_0(y_2)\chi_1(\lambda)d\lambda dy.
\end{eqnarray*}

 If
$
|x_{3}-1|\ll 2^{-j},\, |x_{1}|\ll 2^{-j(1-\kappa_1)}, \, |x_2|\ll
2^{-j(1-\kappa_2)}$,  then the phase is the non-oscillating function, because $\lambda\backsim1$ and
\begin{eqnarray*}\nonumber
(x_3-1)-2^{-j\kappa_1}y_{1}x_{1}-2^{-j\kappa_2}y_{2}x_{2}
-x_{3}\phi(2^{-j\kappa_1}y_1,
2^{-j\kappa_2}y_2)=o(2^{-j})
\end{eqnarray*}
provided the supports of $\chi_0$ is small enough.

Consequently, we have the following lower bound:
\begin{eqnarray*}\nonumber
\|M_{k}u_{j}\|_{L^{p'}}\gtrsim 2^{j\left(2(3-|\kappa|)\left(\frac{1}{p}-\frac{1}{2}\right)-k\right)}.
\end{eqnarray*}

We can choose a linear coordinate system such that $h_{lin}(\phi)=1/|\kappa|$.
Therefore, if
\begin{eqnarray*}
k<k_p(v):=2\left(3-\frac1{h_{lin}(\phi)}
\right)\left(\frac{1}{p}-\frac{1}{2}\right),
\end{eqnarray*}
 then
\begin{eqnarray*}
\|M_{k}u_{j}\|_{L^{p'}}\rightarrow \infty\quad  (\mbox{as}\quad j\to+\infty).
\end{eqnarray*}
 Thus, the operator
 $M_{k}:L^{p}(\mathbb{R}^3)\rightarrow L^{p'}(\mathbb{R}^3)$ is unbounded provided $k<k_p(v)$.

In particular, we obtain the sharp lower bound for the case when the coordinate system is linearly adapted to the function $\phi$.
Thus, we obtain a proof of the Proposition \ref{LAcase}.

Moreover, we receive  a proof of the first part of the Theorem \ref{main} i.e. we get the sharp value of $k_p(v)$ in  the case when $\phi $ has a linearly adapted coordinate system.

\end{proof}

Further, we consider the case
 $2m+1<n$.

\begin{remark}\label{nonadap}
The proof of the Theorem \ref{lowest} shows that if $\phi$ has singularity of type $D_{n+1}$ and  $2m+1<n$ and $k< \left(5-\frac1{2m+1}\right)\left(\frac1p-\frac12\right),$ then
$\|M_{k}u_{j}\|_{L^{p'}}\rightarrow \infty (\mbox{as}\, j\to+\infty)$. Thus, the operator
 $M_{k}:L^{p}(\mathbb{R}^3)\rightarrow L^{p'}(\mathbb{R}^3)$ is an unbounded operator, whenever $k<\left(5-\frac1{2m+1}\right)\left(\frac1p-\frac12\right)$. Because
 \begin{eqnarray*}
 5-\frac1{2m+1}=6-\frac{2(m+1)}{2m+1}=6-\frac2{h_{lin}(\phi)}.
 \end{eqnarray*}
 \end{remark}

Now, we prove the following Theorem.

 \begin{thm}\label{NLA}
 If $2m+1<n$  then
 \begin{equation}\label{NLAcase}
  k_p(v)=\max\left\{ \left(5-\frac1{2m+1}\right)\left(\frac1p-\frac12\right), \left(6-\frac{2m+2}{n}\right)\left(\frac1p-\frac12\right)+\frac{2m+1}{2n}-\frac12\right\}.
 \end{equation}
  \end{thm}

\begin{proof}
We already, proved the upper bound. So, it is enough to prove the inequality
\begin{eqnarray*}
 k_p(v)\ge \max\left\{ \left(5-\frac1{2m+1}\right)\left(\frac1p-\frac12\right), \left(6-\frac{2m+2}{n}\right)\left(\frac1p-\frac12\right)+\frac{2m+1}{2n}-\frac12\right\}.
\end{eqnarray*}
If $k<(5-\frac{1}{2m+1})(\frac{1}{p}-\frac{1}{2})$,
 then the operator  $M_{k}$ is not $L^p(\mathbb{R}^3)\mapsto L^p(\mathbb{R}^3)$ bounded (see Remark \ref{nonadap}).
So, $k_p(v)\ge  (5-\frac{1}{2m+1})(\frac{1}{p}-\frac{1}{2})$.

Further, assume that
\begin{eqnarray*}
k < \left(6-\frac{2m+2}{n}\right)\left(\frac1p-\frac12\right)+\frac{2m+1}{2n}-\frac12.
\end{eqnarray*}
We show that   $M_{k}$ is not $L^p(\mathbb{R}^3)\mapsto L^{p'}(\mathbb{R}^3)$ bounded.

We  slightly  modify  the M. Sugimoto sequence  (see \cite{sugumoto98}) and consider the sequence
\begin{eqnarray*}
u_{j}(x)=2^{-\frac{3j}{p'}+\frac{j(m+1)}{np'}}F^{-1}(v_{j}(2^{-j}\cdot))(x),
\end{eqnarray*}
where
$$v_{j}(\xi)=\chi_0\left(2^{\frac{jm}{n}}\left(
\frac{\xi_2}{\varphi(\xi)}-\left(\frac{\xi_1}{\varphi(\xi)}\right)^{m}\omega
\left(\frac{\xi_1}{\varphi(\xi)}\right)\right)\right)\frac{\chi_1\left(2^{\frac{j}{n}}\frac{\xi_1}
{\varphi(\xi)}\right)\chi_1(\varphi(\xi))|\xi|^{k}}{\varphi^{2}(\xi)
G\left(\frac{\xi_{1}}{\varphi(\xi)},\frac{\xi_2}{\varphi(\xi)}\right)},$$ where
$\chi_0,\,  \chi_1\in C_{0}^{\infty}(\mathbb{R})$ are non-negative smooth functions satisfying the conditions: $\chi_0(0)=1$ and support of
$\chi_0$  lie in a sufficiently small neighborhood of the origin. Suppose $0<c<<1$ is a fixed positive number (say $c=0.0001$) and $\chi_1$ is a non-negative smooth function concentrated in a sufficiently small neighborhood of the point $c$ and identically vanishes in a neighborhood of the origin and also $\chi_1(c)=1$,  (cf. \cite{sugumoto98}). Obviously
$v_{j}\in C_{0}^{\infty}(\mathbb{R}^{3})$. We will estimate the $L^p(\mathbb{R}^3)-$ norm of the function $u_j$:
We have
\begin{eqnarray*}
u_{j}(x)=2^{\frac{3j}{p}+\frac{j(m+1)}{np'}}\int_{\mathbb{R}^3}e^{-i2^j(\xi, x)}v_j(\xi)d\xi.
\end{eqnarray*}
 As before, we use change of variables $\xi=\lambda(y_{1}, y_{2}, 1+\phi(y_{1},y_{2}))$, then we use change of variables:
\begin{eqnarray*}
y_1=2^{-\frac{j}n} \eta_1, \quad y_2=2^{-\frac{jm}n}(\eta_2+\eta_1^m\omega(2^{-\frac{j}n} \eta_1)).
\end{eqnarray*}
Then we get:
\begin{eqnarray*}
u_{j}(x)=2^{\frac{3j}{p}-\frac{j(m+1)}{np}}\int_{\mathbb{R}^3}e^{-i2^{j}\lambda\Phi(\eta, x, 2^{-j})}
\chi_1(\eta_1)\chi_0(\eta_2)\\ \chi_1(\lambda)\lambda^k(2^{-\frac{2j}n} \eta_1^2+2^{-\frac{2jm}n}(\eta_2+\eta_1^m\omega(2^{-\frac{j}n} \eta_1))^2+(1+\tilde\phi(\eta))^2)^\frac{k}2 d\lambda d\eta,
\end{eqnarray*}
where
\begin{eqnarray*}
\Phi(\eta, x, 2^{-j}):= x_3(1+2^{-\frac{2m+1}n j} \eta_2^2(\eta_1 b_1(\delta_{2^{-j}}(\eta))+\eta_2^22^{\frac{1-2m}n j}b_2(2^{-\frac{mj}n} \eta_2))+ \\ 2^{-j} \eta_1^n\beta(2^{-\frac{j}n} \eta_1)))+  2^{-\frac{j}n} \eta_1x_1+2^{-\frac{jm}n}(\eta_2+\eta_1^m\omega(2^{-\frac{j}n} \eta_1))x_2.
\end{eqnarray*}

If $|x_3|>> |2^{-\frac{j}n} x_1|+|2^{-\frac{jm}n}x_2| $ then we can use integration by parts formula in $\lambda$ and obtain:
\begin{eqnarray*}\nonumber
  |u_{j}(x)|\lesssim_N \frac{2^{\frac{3j}p-\frac{(m+1)j}{np}}}{(1+|x_32^j|)^N},
\end{eqnarray*}
provided $|x_32^j|>>1$, otherwise the last estimate for the function $u_j(x)$ is trivially holds.

Then we consider the case $|x_3|<<|2^{j-\frac{j}n} x_1|+|2^{j-\frac{jm}n}x_2| $. Then we can use integration by parts in $(\eta_1, \eta_2)$ to have the estimate:
\begin{eqnarray*}\nonumber
  |u_{j}(x)|\lesssim_N \frac{2^{\frac{3j}p-\frac{(m+1)j}{np}}}{(1+|2^{j-\frac{j}n} x_1|+|2^{j-\frac{jm}n}x_2|)^N}.
\end{eqnarray*}
Now, we assume $|x_3|\backsim|2^{-\frac{j}n} x_1|+|2^{-\frac{jm}n}x_2|$. Moreover, if $|2^{-\frac{j}n} x_1|\not\backsim|2^{-\frac{jm}n}x_2| $.
Then we obtain:
\begin{equation}\label{intpart}
  |u_{j}(x)|\lesssim_N \frac{2^{\frac{3j}p-\frac{(m+1)j}{pn}}}{(1+|2^jx_3|+|2^{j-\frac{j}n} x_1|+|2^{j-\frac{jm}n}x_2|)^N}.
\end{equation}

Finally, we consider the case $|x_3|\backsim|2^{-\frac{j}n} x_1|\backsim|2^{-\frac{jm}n}x_2| $.
Then the phase function has no critical points in $\eta_2$. Then we can obtain estimate \eqref{intpart} by using integration by parts in $\eta_2$.

Thus, due to the inequality \eqref{intpart}  for large  $j$ we have
  $$\|u_{j}\|_{L^{p}(\mathbb{R}^{3})}\sim 1.$$
  Now, we consider a lower bound for $\|M_{k}u_{j}\|_{L^{p'}(\mathbb{R}^{3})}.$

We have:
$$M_{k}u_{j}=F^{-1}e^{i\varphi(\cdot)}a_{k}(\cdot)Fu_{j}=2^{-\frac{3j}{p'}
+j\frac{m+1}{np'}}F^{-1}(e^{i\varphi(\cdot)}a_{k}(\cdot)v_{j}(2^{-j}\cdot))(x).$$

We perform  change of variables given by the scaling  $2^{j}\xi\rightarrow \xi$ and obtain:
\begin{eqnarray*}\nonumber
M_{k}u_{j}(x)=\frac{2^{\frac{3j}{p}+\frac{j(m+1)}{np'}-kj}}{\sqrt{(2\pi)^3}}\int_{\mathbb{R}^{3}}
e^{i2^{j}(\varphi(\xi)-\xi
x)}\\ \chi_0\left(2^{\frac{jm}{n}}\left(\frac{\xi_2}{\varphi(\xi)}-\left(\frac{\xi_1}
{\varphi(\xi)}\right)^{m}\omega\left(\frac{\xi_1}{\varphi(\xi)}\right)\right)\right) \frac{\chi_0\left(2^{\frac{j}{n}
}\frac{\xi_1}{\varphi(\xi)}\right)\chi_1(\varphi(\xi))}{\varphi^{2}(\xi)
G\left(\frac{\xi_1}{\varphi(\xi)},\frac{\xi_2}{\varphi(\xi)}\right)}d\xi.
\end{eqnarray*}

Finally, we use change of variables  $\xi\rightarrow
\lambda(y_1,y_2,1+\phi(y_1,y_2))$.  Then we have:
\begin{eqnarray*}\nonumber
M_{k}u_{j}(x)=\frac{2^{\frac{3j}{p}+\frac{j(m+1)}{p'}-kj}}{\sqrt{(2\pi)^3}}\int_{\mathbb{R}^3}
e^{i2^{j}\lambda(1-
x_{3}-(y_{1}x_{1}+y_{2}x_{2}+x_{3}\phi(y_{1},y_{2})))}\times\\ \times
\chi_0(2^{\frac{jm}{n}}(y_{2}-y_{1}^{m}\omega(y_1)))\chi_1(2^{\frac{j}{n}}y_{1})
\chi_1(\lambda)d\lambda dy_1 dy_2.
\end{eqnarray*}

Now, we perform the change of variables
 \begin{eqnarray*}
 y_{1}=2^{-\frac{j}{n}}z_{1}, \quad y_2=2^{-j\frac{m}{n}}z_{1}^{m}\omega
 (2^{-\frac{j}{n}}z_{1})+2^{-j\frac{m}{n}}z_2.
 \end{eqnarray*}
  Then we get
  \begin{eqnarray*}\nonumber
  M_{k}u_{j}(x)=2^{\frac{3j}{p}+\frac{m+1}{np'}j-\frac{m+1}{n}j-kj}
   \int e^{i2^{j}\lambda\Phi_3(z, x, j)}\chi_0(z_2)\chi_1(z_1)\chi_1(\lambda)d\lambda dz_1 dz_2,
  \end{eqnarray*}
where
\begin{eqnarray*}\nonumber
\Phi_3(z, x, j):=1- x_3 -(2^{-\frac{j}{n}}x_1 z_1 +
  x_2 2^{-\frac{jm}{n}}z_{1}^{m}\omega(2^{-\frac{j}{n}}z_1)+z_2 2^
  {-\frac{jm}{n}}x_2+\\ x_3 2^{-\frac{j(2m+1)}{n}}z_1z_{2}^{2}b(2^{-\frac{j}{n}}z_1,
  2^{-\frac{jm}{n}}(z_{1}^{m}\omega(2^{-\frac{j}{n}}z_1)+z_2))+2^{-j}z_{1}^{n}
  \beta(2^{-\frac{j}{n}}z_1)).
 \end{eqnarray*}

We use stationary phase method in $z_2$ assuming , $|1-x_3|<<2^{-j}, \,
  |x_1|<<2^{-\frac{n-1}{n}j}, \, |x_2|<<2^{-\frac{j(n-m)}{n}}$. We can use stationary phase method in $z_2$ because $z_1\backsim1$. Then we obtain:
  \begin{eqnarray*}\nonumber
  M_k u_{j}(x)=2^{j(\frac{3}{p}+\frac{m+1}{np'}-\frac1{2n}-\frac12
  -k)}
  \left(\int_{\mathbb{R}^{2}}e^{i2^{j}\lambda \Phi_4}\chi_0(z_{2}^{c}(z_1, x_2))\chi_1(z_1)\chi_1(\lambda)d\lambda dz_1+O(2^{j(\frac{2m+1}n-1))}\right),
\end{eqnarray*}
where
  \begin{eqnarray*}\nonumber
\Phi_4:=\Phi_4(z_1, x, j):=1-x_3 -x_1 z_1
  2^{-\frac{j}{n}}+x_2 2^{-\frac{jm}{n}}z_{1}^{m}\omega(2^{-\frac{j}{n}}z_1
  )+ \\ 2^{-j}z_{1}^{n}\beta(2^{-\frac{j}{n}}z_1)+x_{2}^{2}2^{-\frac{2jm}{n}}
  B(z_1,x_2).
\end{eqnarray*}

From here we obtain the lower bound:
  $$\|M_{k}u_{j}\|_{L^{p'}(\mathbb{R}^{3})}\ge 2^{j\left(\left(6-\frac{2m+2}{n}\right)\left(\frac1p-\frac12\right)+\frac{2m+1}{2n}-\frac12-k\right)}c,$$
   where $c>0$ is a constant which does not depend on $j$.
   Thus if
   \begin{eqnarray*}
   k<\left(6-\frac{2m+2}{n}\right)\left(\frac1p-\frac12\right)+\frac{2m+1}{2n}-\frac12
   \end{eqnarray*}
   then the operator $M_k$ is not $L^p(\mathbb{R}^3)\mapsto L^{p'}(\mathbb{R}^3)$ bounded.

Analogical result holds true for the case $n=\infty$.

Thus if  $k<k_p(v)$ then the   $M_k$ is not $L^p- L^{p'}$ bounded operator.
   This completes a proof of the Theorem \ref{NLA}.

\end{proof}

A proof of the  main Theorem \ref{main}  follows from the Theorems \ref{NLA}  and \ref{preest} with \ref{lowest}.



\end{document}